\documentclass[a4paper,12pt,leqno]{amsart}

\usepackage[active]{srcltx}
\usepackage{verbatim}
\usepackage{epsfig,graphicx,color,mathrsfs}
\usepackage{graphicx}
\usepackage{amsmath,amssymb,amsthm,amsfonts}
\usepackage{amssymb}
\usepackage[english]{babel}

\newtheorem{thm}{Theorem}[section]
\newtheorem{cor}[thm]{Corollary}
\newtheorem{lem}[thm]{Lemma}
\newtheorem{prop}[thm]{Proposition}
\newtheorem{rem}[thm]{Remark}
\newtheorem{definition}[thm]{Definition}

\newcommand{\bremark}{\begin{rem} \textup}
\newcommand{\eremark}{\end{rem} }

\newcommand{\pa}{\partial}

\newcommand{\cuad}{{\sqcap\kern-.68em\sqcup}}

\newcommand{\R}{{\mathbb{R}}}

\numberwithin{equation}{section}

\thanks{
AF is partially supported  by ERC-2011-grant: \emph{Epsilon} and by ERC-2013-grant: \emph {COMPAT}\\
BS is partially supported by the
Italian PRIN Research Project 2007: {\em Metodi Variazionali e Topologici
nello Studio di Fenomeni non Lineari}, and is also partially supported by  ERC-2011-grant: \emph{Epsilon}.\\
}
\begin{document}

\parindent 0pc
\parskip 6pt
\overfullrule=0pt

\title[Farina-Sciunzi-2 ]{Monotonicity and symmetry \\
of nonnegative solutions to \\  $ -\Delta u=f(u) $ \\ in  half-planes and strips}


\author[
A. Farina and B. Sciunzi]{A. Farina and B. Sciunzi}
\thanks{Address: {\em AF} --
 LAMFA, CNRS UMR 7352, Universit\'{e} de Picardie Jules Verne, 33, Rue Saint-Leu, 80039 Amiens Cedex 1, France.
 E-mail: {\tt
alberto.farina@u-picardie.fr}. {\em BS} --
Universit\`a della Calabria -- (Dipartimento di Matematica e Informatica) -- V. P.
 Bucci 1-- Arcavacata di Rende (CS), Italy. E-mail: {\tt
 sciunzi@mat.unical.it}.\\}

\maketitle
\begin{flushright}
\emph{A Ireneo con profonda stima e amicizia.}
\end{flushright}

\begin{abstract}
We consider nonnegative solutions to $-\Delta u=f(u)$ in half-planes and strips, under zero Dirichlet boundary condition. Exploiting a rotating$\&$sliding line technique, we prove symmetry and monotonicity properties of the solutions, under very general assumptions on the nonlinearity $f$. In fact we provide a unified approach that works in all the cases $f(0)<0$, $f(0)= 0$ or $f(0)> 0$.
Furthermore we make the effort to deal with nonlinearities $f$ that may be not locally-Lipschitz continuous.
We also provide explicite examples showing the sharpness of our assumptions on the nonlinear function $f$.
\end{abstract}

\section{Introduction and main results} \label{Sec1}

We consider the problem of classifying solutions to
\begin{equation}\label{E:P}
\begin{cases}
-\Delta u=f(u) & \text{ in } \mathbb{R}^2_+,\\
\quad u \ge  0 & \text{ in } \mathbb{R}^2_+,\\
\quad u=0\,\, &\text{ on } \partial\mathbb{R}^2_+.
\end{cases}
\end{equation}
under very general assumptions on the nonlinear function $f$. Here, by $ \mathbb{R}^2_+$, we mean the open half-plane $\{ (x,y)  \in \R^2 : y >0\}$.



Let us recall that, in our two previous works \cite{DS3,FS0}, we proved the following result.

\begin{thm} \label{teo-doppio}
Let $u\in C^2(\overline{\R^2_+})$ be a solution to
\begin{equation}\label{P1-2D}
\begin{cases}
-\Delta u = f(u)& \text{in $\,\, \R^2_+$},\\
u\ge0& \text{in $\,\, \R^2_+$},\\
u=0 & \text{on $\,\, \pa \R^2_+$},
\end{cases}
\end{equation}
with $f$ {\rm locally Lipschitz continuous} on $[0\,, +\infty)$.  Then

\item[i)] if $f(0) \ge 0$, either $u$ vanishes identically, or $u$ is positive on $\R^2_+$ with $ \frac{\partial u}{\partial y} >0$ in $\R^2_+$.

\item[ii)]  if $ f(0) <0$, either $u$ is positive on $\R^2_+$, with  $ \frac{\partial u}{\partial y} >0$ in $\R^2_+$,
or $u$ is one-dimensional and periodic (and unique).

\end{thm}

The case $f(0) \ge 0$ was first treated in \cite{DS3} (in which also the case of the $p$-Laplace operator is considered) while the case $f(0) <0$ was carried out later, in the  paper \cite{FS0}.  Both of them are  based on a refined version of the \emph{moving plane method} \cite{S} (see also \cite{BN,GNN}). More precisely, the authors of \cite{DS3} exploit a
\emph{rotating line technique} and a \emph{sliding line technique}, while in \cite{FS0} we have used a refinement of these techniques combined with the \emph{unique continuation principle}, needed to handle the new, challenging and difficult case of nonnegative solutions.
The techniques developed in \cite{DS3,FS0} also provided an affirmative answer to a conjecture and to an open question posed by Berestycki, Caffarelli and Nirenberg in \cite {BCN1, BCN2}.




\noindent In this paper we first give a (new) unique/unified proof to  Theorem \ref{teo-doppio} and at the same time we make the effort to deal with the case of continuous nonlinearities $f$ that fulfills very weak and general regularity assumptions, i.e., less regular than locally Lipschitz continuous. See assumptions ($h_1$)-($h_4$) in Section 2.

In this direction we have the following results.


\begin{thm}\label{casteorem}
Let $u\in C^2(\overline{\mathbb{R}^2_+})$ be a solution to
 \begin{equation}\label{P1-2D-pos}
\begin{cases}
-\Delta u = f(u)& \text{in $\,\, \R^2_+$},\\
u>0& \text{in $\,\, \R^2_+$},\\
u=0 & \text{on $\,\, \pa \R^2_+$},
\end{cases}
\end{equation}
and assume that either

\item[i)] $f(0) =0$ and $f$ fulfills ($h_1$),  ($h_2$) and ($h_3$),

or

\item[ii)] $f(0)\neq0$ and $f$ fulfills ($h_1$),  ($h_2$) and ($h_4$).

Then
\[
\frac{\partial u}{\partial y} >0 \qquad\text{in }\quad\mathbb{R}^2_+\,.
\]
\end{thm}

\bigskip

We observe that the conclusion of  Theorem \ref{casteorem} is not true if we drop the assumption $(h_3)$ in item i).  This will be discussed in the last section of the paper.

\bigskip

Let us point out that, with the same technique, we can also prove a symmetry and monotonicity result in strips, for possibly unbounded solutions.
More precisely, with the notation $\Sigma_{2b}\,:= \{ (x,y)  \in \R^2 : y \in  (0\,,\,2b) \}$, $b >0$, we have the following:
\begin{thm}\label{T:3spe}
Let $u\in C^2(\overline{\Sigma_{c}})$ for any $c<2b$ be a   solution to
\begin{equation}\label{E:Pstrip}
\begin{cases}
-\Delta u=f(u), & \text{ in } \Sigma_{2b}\\
\quad u > 0, & \text{ in } \Sigma_{2b}\\
\quad u=0,\,\, &\text{ on } \{y=0\} \,
\end{cases}
\end{equation}

and assume that either

\item[i)] $f(0) =0$ and $f$ fulfills ($h_1$),  ($h_2$) and ($h_3$),

or

\item[ii)] $f(0)\neq0$ and $f$ fulfills ($h_1$),  ($h_2$) and ($h_4$).

 Then
 \[
\frac{\partial u}{\partial y} >0 \qquad\text{in }\quad\Sigma_{b}\,.
\]
If $u\in C^2(\overline{\Sigma_{2b}})$
and $u=0$ on $\partial \Sigma_{2b}$, then
 $u$ is  symmetric about $\{y=b\}$.
\end{thm}

\bigskip

As already observed for Theorem \ref{casteorem}, the conclusion of Theorem \ref{T:3spe}  is not true if we drop the assumption $(h_3)$ in item i) (see section 6 of this paper).

\medskip

The  Theorem above complements Theorem 1.3 of \cite{FS0} and also provides an affirmative answer to an (extended version of an) open question posed by Berestycki, Caffarelli and Nirenberg in \cite {BCN1}.

\medskip

Next we prove a symmetry result in the case of the half-plane.

\begin{thm} \label{teo-sim}
Let $u\in C^2(\overline{\R^2_+})$ be a solution to
\begin{equation}
\begin{cases}
-\Delta u = f(u)& \text{in $\,\, \R^2_+$},\\
u\ge0& \text{in $\,\, \R^2_+$},\\
u=0 & \text{on $\,\, \pa \R^2_+$},
\end{cases}
\end{equation}
with $f$ {\rm locally Lipschitz continuous} on $[0\,, +\infty)$ and satisfying
\begin{equation}
\exists p >1 \quad : \quad  \lim_{t\to+\infty} \frac{f(t)}{t^p} =l \in (0,+\infty).
\end{equation}

Then $u$ is bounded and one-dimensional, i.e.
\[
u(x,y)\,=\,u_0(y)\, \qquad \forall \, (x,y) \in \R^2_+,
\]
for some bounded function $u_0 \in C^2([0, +\infty))$.
\end{thm}

As a consequence of the  results above we also obtain the following


\begin{cor}\label{CorClas}
Let $f$ be {\rm locally Lipschitz continuous} on $[0\,, +\infty)$ and satisfying
\begin{equation}\label{posit}
f(t)>0 \qquad \forall \, t>0,
\end{equation}
\begin{equation}
\exists \, p >1 \quad : \quad  \lim_{t\to+\infty} \frac{f(t)}{t^p} =l \in (0,+\infty).
\end{equation}
Then
\item[i)] if $f(0) =0$, the only solution of class $ C^2(\overline{\R^2_+})$ of
\begin{equation}\label{pblA}
\begin{cases}
-\Delta u = f(u)& \text{in $\,\, \R^2_+$},\\
u\ge0& \text{in $\,\, \R^2_+$},\\
u=0 & \text{on $\,\, \pa \R^2_+$},
\end{cases}
\end{equation}
is $ u \equiv 0$,

\item[ii)] if $f(0) >0$, problem \eqref{pblA} has no solution of class $ C^2(\overline{\R^2_+})$.

\end{cor}

We conclude this section with the following classification result

\begin{thm}\label{teo-clas2}
Let $f$ be  non decreasing {\rm locally Lipschitz continuous} on $[0\,, +\infty)$ satisfying
\begin{equation}\label{positzero}
f(0) \ge 0
\end{equation}

Then, the problem
\begin{equation}\label{pblB}
\begin{cases}
-\Delta u = f(u)& \text{in $\,\, \R^2_+$},\\
u\ge0& \text{in $\,\, \R^2_+$},\\
u=0 & \text{on $\,\, \pa \R^2_+$},
\end{cases}
\end{equation}
has a nontrivial solution of class  $ C^2(\overline{\R^2_+})$ if and only if $f \equiv 0$. In the latter case $u$ is necessarily linear, i.e., $ u(x,y) = cy$, for some constant $c > 0$.

Furthermore, when $ f \not \equiv 0$, we have that

\item[i)] if $f(0) =0$, the only solution of class $ C^2(\overline{\R^2_+})$ of \eqref{pblB} is $ u \equiv 0$.

\item[ii)] if $f(0) >0$, problem \eqref{pblB} has no solution of class $ C^2(\overline{\R^2_+})$.

\end{thm}

\begin{rem}
The assumption \eqref{positzero} is sharp. Indeed, the function $u(x,y)=\,1-\cos y$ is a nontrivial solution of
\begin{equation}
\begin{cases}
-\Delta u=u-1 & \text{ in } \mathbb{R}^2_+,\\
\quad u \geqslant 0 & \text{ in } \mathbb{R}^2_+,\\
\quad u=0\,\, &\text{ on } \partial\mathbb{R}^2_+,
\end{cases}
\end{equation}
and $f(0) = -1 <0$.
\end{rem}

\begin{rem}

\item[i)] Theorem \ref{teo-clas2} applies, for instance, to the functions $ f(u) = u^p + c$, with $ p\ge 1 $ and $ c \ge 0$. In particular, for $c=0$ (i.e. for $f(u) =u^p$) we obtain a new and different proof  of a celebrated result of Gidas and Spruck \cite{GiSp}  (see also \cite{DS3}).

\item[ii)] Note also that one can apply item i) of Corollary \ref{CorClas} to $f(u) =u^p$, $p>1$, to obtain another new and different proof of the above mentioned result of Gidas and Spruck \cite{GiSp}.

\end{rem}

\bigskip

In this work we focused on the two-dimensional case and we provided a precise description of the situation under very general assumptions (both on $f$ and on $u$). Differently from the two-dimensional case, the situation is not yet well-understood for dimensions $N \ge 3$. For results in the higher dimensional case (and with additional assumptions on $f$ and on $u$)
we refer to \cite{BCN1,BCN2,BCN3,BCN5,CLZ,CMS,melian,DS3, Dancer,Dancer2,EL,Fa1,Fa,Fa2,FS0,FSV,soave,FV,FVb,GiSp}.

\section{Assumptions and preliminary results} \label{preliminaries}

We start discussing the main assumptions on the nonlinearity $f$. In Theorem \ref{teo-doppio} we assume that
 $f$ is {\rm locally Lipschitz continuous} on $[0\,, +\infty)$. This allow us to deal with nonnegative solutions. In Theorem \ref{casteorem} we restrict our attention to positive solutions weakening the assumptions on the nonlinearity. It is convenient to set the following

\noindent ($h_1$) $f$ is continuous in $[0\,, +\infty)$.\\
\noindent ($h_2$) $f$ is {\it locally Lipschitz continuous from above} in $[0\,,\,\infty)$, i.e., for every $ b>0$ there is $L_b>0$ :
\begin{equation}\nonumber
f(u)-f(v)\leq L_b (u-v)\qquad \text{for any}\quad0\leq v\leq u \le b.
\end{equation}\\
\noindent ($h_3$)
For any $\bar t\in [0\,, +\infty)$, there exists $\delta=\delta(\bar t)>0$ such that
\begin{equation}\label{prop-conf-forte}
f(s)-f(t)\leq g(t-s)\,\qquad \forall \,s\leq t\in [\bar t- \frac{\delta}{2}\,,\,\bar t+\frac{\delta}{2}] \cap [0\,, +\infty)
\end{equation}
with $g\in C^0[0\,,\delta]$, $g(0)=0$ and, either $g$ vanishes identically in $[0,\delta]$, or it is  positive and non decreasing in $[0,\delta]$ with
\begin{equation}\label{cond-int-PM}
\int_0^\delta \frac{ds}{\sqrt{G(s)}}=\infty
\end{equation}
where $G(s)\,:=\, \int_0^s\,g(t)\,dt$. \\
\noindent ($h_4$) Condition  ($h_3$) holds but only for any $\bar t\in (0\,, +\infty)$.

\bigskip

Note that the hypothesis $(h_2)$ is very weak (actually it is not even enough to ensure the continuity of $f$) and it is clearly satisfied by any non increasing function. Actually, all the nonlinearities of the form
\[
 f(s)\,:=\,f_1(s)+f_2(s)\,,
 \]
for some non increasing continuous function $f_1(\cdot)$ in $[0,\infty)$ and some $f_2(\cdot),$ which is locally Lipschitz continuous in $[0,\infty)$, satisfy both $(h_1)$ and $(h_2)$.

We also observe that the assumption $(h_3)$ is natural. Indeed, the conditions imposed on $g$ are the well-known optimal assumptions ensuring the validity of the strong maximum principle and the Hopf's lemma (see \cite{PSB} for instance).  It is also clear that every locally Lipschitz continuous function on $[0\,,\,\infty)$ satisfies $(h_1)$,$(h_2)$ and $(h_3)$ with $g(t)= L_{\delta} t$, $L_{\delta}>0$ being any constant larger than the Lipschitz constant of $f$ on the interval $[0,\delta]$.  On the other hand, the converse is not true. This is the case e.g.  when
 $f(\cdot)$ has the form
 \[
 f(s)\,:=\,g(s)+c\,,
 \]
where $g(s)\equiv s  \,\log  (s)$ in some interval $(0,\delta)$, $\delta>0$ and $g(\cdot)$ of class $C^1$ in $(0\,,\,\infty)$.
It is easy to verify that such a nonlinearity  fulfills ($h_3$) but it is not Lipschitz continuous at zero.


\medskip

Now we are ready to prove

\medskip

\begin{prop}[Weak Comparison Principle in domains of small measure]\label{maxpri}
Let $ N\ge 1$, assume that $f$ fulfills ($h_2$) and fix a real number $k>0$. Then there exists $\vartheta\,=\,\vartheta(N,k,f)>0$ such that, for any domain $D \subset \R^N$, with $\mathcal L (D)\leq \vartheta$, and any $u,v\in H^1(D) \cap C^0(\overline D)$ such that
\begin{equation}\label{eqdifferenza}
\begin{cases}
-\Delta u-f(u)\leq -\Delta v-f(v)& \qquad\text{in}\quad D, \\
0 \le u, v \le k & \qquad\text{in}\quad D,\\
u \le v & \qquad \text{on}\quad \partial D,
\end{cases}
\end{equation}
then


\[
u\leq v\qquad\text{in}\,\,D\,.
\]
\end{prop}
\begin{proof}
We use $(u-v)^+\in H^1_0(D)$ as test function in the weak formulation of \eqref{eqdifferenza} and get
\begin{equation}\nonumber
\begin{split}
\int_{D} \big|\nabla (u-v)^+\big|^2\,dx&\leq \int_{D} \, (f(u)-f(v)) (u-v)^+ \,dx\\
&\leqslant L_k\int_{D} \,((u-v)^+)^2\,dx
\end{split}
\end{equation}

where $L_k$ is the positive constant appearing in $(h_2)$ and corresponding to $ b = k>0$. Note that $L_k$ depends only on $k$ and $f$.

An application of Poincar\'e inequality gives

\begin{equation}\nonumber
\int_{D} \big|\nabla (u-v)^+\big|^2\,dx \leqslant L_k(C_N (\mathcal L (D'))^{\frac{2}{N}}) \int_{D} \big|\nabla (u-v)^+\big|^2\,dx\,,
\end{equation}

where $C_N>0$ is a constant depending only on the euclidean dimension $N$.

\noindent The desired conclusion then follows by choosing $\vartheta = \frac{1}{2 L_k^{\frac{N}{2}} C_N}$.  Indeed, the latter implies that
$ L_k (C_N (\mathcal L (D))^{\frac{2}{N}}) <1 $ and so we get that $(u-v)^+\equiv 0$ and the thesis.
\end{proof}

\bigskip

Now we focus on the two-dimensional case and fix some notations.

Given $x_0 \in \R$, $s>0$ and $\theta\in(0\,,\,\frac{\pi}{2})$, let $L_{x_0,s,\theta}$ be the line, with slope $\tan(\theta)$, passing through $(x_0,s)$. Also, let
$V_\theta$ be the vector orthogonal to $L_{x_0,s,\theta}$ such that $(V_\theta,e_2) >0$ and $\|V_\theta\|=1$.

We denote by
\begin{equation}\label{rev1}
\mathcal{T}_{x_0,s,\theta},
\end{equation}
 the (open) triangle delimited by $L_{x_0,s,\theta}$, $\{y=0\}$ and $\{x=x_0\}$, and we define

$$
u_{x_0,s,\theta}(x)=u(T_{x_0,s,\theta}(x)), \quad x \in \mathcal{T}_{x_0,s,\theta}
$$
where $T_{x_0,s,\theta}(x)$ is the point symmetric to $x$, w.r.t. $L_{x_0,s,\theta}$.  It is immediate to see that
$u_{x_0,s,\theta}$ still fulfills $-\Delta u_{x_0,s,\theta}=f(u_{x_0,s,\theta})$ on the triangle $\mathcal{T}_{x_0,s,\theta}$.

We also consider
\begin{equation}\label{fskhkhgfjhgjhfj}
w_{x_0,s,\theta}=u-u_{x_0,s,\theta} \qquad on \quad \mathcal{T}_{x_0,s,\theta}
\end{equation}
and observe that
\begin{equation}\label{conf-forte}
w_{x_0,s,\theta} \le 0 \quad on \quad \mathcal{T}_{x_0,s,\theta} \quad \Longrightarrow \quad either \quad w_{x_0,s,\theta} \equiv 0 \quad or \quad w_{x_0,s,\theta} < 0 \quad  on \quad \mathcal{T}_{x_0,s,\theta}
\end{equation}
thanks to the assumption ($h_3$) (resp. to ($h_4$), if $u$ is supposed to be positive).  Indeed, if $ \overline {x} \in \mathcal{T}_{x_0,s,\theta}$ is such that  $w_{x_0,s,\theta} (\overline{x}) = 0 $ then, by the continuity of $u$ and of $ u_{x_0,s,\theta} $ we can find an open ball centered at $\overline{x}$, say $B_{\overline{x}} \subset \mathcal{T}_{x_0,s,\theta}$, such that

\begin{center}

$\forall \, x \in B_{\overline{x}} \quad (u(\overline{x}) - \frac{\delta}{2})^+ \le u(x) \le u(\overline{x}) + \frac{\delta}{2}$,

$\forall \, x \in B_{\overline{x}} \quad (u(\overline{x}) - \frac{\delta}{2})^+ \le u_{x_0,s,\theta} (x) \le u(\overline{x}) + \frac{\delta}{2}$,

\end{center}

where $  \overline{t} = u(\overline{x}) = u_{x_0,s,\theta} (\overline{x})$ and $ \delta =\delta(\overline{t}) >0$ is the one provided by the assumption ($h_3$) (resp. ($h_4$)).
Now, since $w_{x_0,s,\theta} \le 0$ on $\mathcal{T}_{x_0,s,\theta} $, we can apply \eqref{prop-conf-forte} to get


\begin{equation}\label{eq-conf-forte}
\begin{cases}
\Delta ( - w_{x_0,s,\theta} ) \leq g(- w_{x_0,s,\theta} )& \qquad\text{in}\quad B_{\overline{x}}, \\
- w_{x_0,s,\theta} \ge 0 & \qquad\text{in}\quad B_{\overline{x}}, \\
- w_{x_0,s,\theta} (\overline{x}) = 0
\end{cases}
\end{equation}
and the strong maximum principle (see \cite{PSB} for instance) yields $w_{x_0,s,\theta}  \equiv  0$ on $B_{\overline{x}}$. After that, a standard  connectedness argument provides $w_{x_0,s,\theta}  \equiv  0$ on the entire triangle $\mathcal{T}_{x_0,s,\theta}$.








\bigskip

In what follows we shall make repeated use of a refined version of the \emph{moving plane technique} \cite{S} (see also \cite{BN,GNN}). Actually we will exploit a
\emph{rotating plane technique} and a \emph{sliding plane technique}  developed in \cite{DS3,FS0}.

Let us give the following definition

\begin{definition}\label{defcondition} Given $x_0, s$ and $\theta $ as above, we say that the condition  $(\mathcal H\mathcal T_{x_0,s,\theta})$ holds in the triangle $\mathcal T_{x_0,s,\theta}$ if
\begin{center}

$\quad w_{x_0,s,\theta}< 0 \quad $ in $\quad \mathcal{T}_{x_0,s,\theta}$,

$w_{x_0,s,\theta}\leqslant 0 \quad $ on $ \quad \partial(\mathcal{T}_{x_0,s,\theta}) \quad $ and

$ w_{x_0,s,\theta}$ is not identically zero on $\partial(\mathcal{T}_{x_0,s,\theta})$,
\end{center}

with $w_{x_0,s,\theta}$ defined in \eqref{fskhkhgfjhgjhfj}.
\end{definition}

We have the following

\begin{lem}[Small Perturbations]\label{smallperturbations}
Let $u\in C^2(\overline{\R^2_+})$ be a {\rm nonnegative} solution to \eqref{E:P}
 and assume that $f$ fulfills ($h_2$) and  ($h_3$).
Let  $(x_0,s,\theta)$  and $\mathcal{T}_{x_0,s,\theta}$ be  as above and assume that
$(\mathcal H\mathcal T_{x_0,s,\theta})$ holds. Then there exists $\bar\mu = \bar\mu(x_0, s, \theta)>0$ such that
\begin{equation}\label{claimsmallpert}
\begin{cases}
\vert \theta - \theta' \vert + \vert s-s'\vert < \bar\mu, &\\
w_{x_0,s',\theta'}\leqslant 0 \quad \text {on} \,\, \partial(\mathcal{T}_{x_0,s',\theta'}), \hskip4truecm \Longrightarrow \quad (\mathcal H\mathcal T_{x_0,s',\theta'}) \,\, \text {holds}.&\\
w_{x_0,s',\theta'} \quad \text {is not identically zero on} \,\, \partial(\mathcal{T}_{x_0,s',\theta'}) &\\
\end{cases}
\end{equation}
If $u$ is positive, the same result holds assuming only ($h_4$) instead of ($h_3$).
\end{lem}


\begin{proof}
Let $R>0$ and $\tilde \mu$ be fixed so that
\begin{equation}\nonumber
\underset{\vert \theta - \theta' \vert + \vert s-s'\vert < \tilde\mu}{\bigcup}\mathcal{T}_{x_0,s',\theta'}\cup T_{x_0,s',\theta'}(\mathcal{T}_{x_0,s',\theta'})\subset B_R(x_0)\,.
\end{equation}
Then we set
\begin{equation}\nonumber
k\,:=\, \underset{\overline {B_R^+(x_0)}}{\max}\,u\,\qquad\text{and}\qquad \vartheta\,=\,\vartheta(N,k,f)
\end{equation}
where $\vartheta(N,k,f)$ is the one appearing in Proposition \ref{maxpri} and $B_R^+(x_0)=B_R(x_0)\cap\{y>0\}$.\\

Now we fix $0<\hat \mu\leq\tilde\mu$ such that
\begin{equation}\label{ppp1}
\mathcal L (\mathcal{T}_{x_0,s+\hat \mu,\theta-\hat\mu}\setminus\mathcal{T}_{x_0,s-\hat \mu,\theta+\hat\mu} ) < \frac{\vartheta}{2}
\end{equation}
and note that
\begin{equation}\nonumber
\mathcal{T}_{x_0,s-\hat \mu,\theta+\hat\mu}\subset
\underset{\vert \theta - \theta' \vert + \vert s-s'\vert < \hat\mu}{\bigcup}\mathcal{T}_{x_0,s',\theta'}\subset \mathcal{T}_{x_0,s+\hat \mu,\theta-\hat\mu} \,.
\end{equation}
Now we can  consider a compact set $K\subset \mathcal{T}_{x_0,s-\hat \mu,\theta+\hat\mu}$ so that
\begin{equation}\label{ppp2}
 \mathcal{L} (\mathcal{T}_{x_0,s-\hat \mu,\theta+\hat\mu} \setminus K) < \frac{\vartheta}{2}\,.
\end{equation}
By assumption we know that $w_{x_0,s,\theta}< 0$ in $\mathcal{T}_{x_0,s,\theta}$  and consequently in the compact set $K$.  Therefore,
by a uniform continuity argument, for some
 $$0<\bar \mu\leq\hat\mu,$$
  we can assume that
  \begin{equation}\label{ppp3}
 w_{x_0,s',\theta'} < 0\quad\text{in}\,\,\,K\qquad \text{for }\quad\vert \theta - \theta' \vert + \vert s-s'\vert < \bar\mu\,.
\end{equation}
  By   \eqref{ppp1} and \eqref{ppp2}  we deduce that
\begin{equation}\nonumber
 \mathcal{L} (\mathcal{T}_{x_0,s',\theta'} \setminus K) < \theta
 \end{equation}
 and,
observing that $w_{x_0,s',\theta'}\leqslant  0$ on $\partial \big(\mathcal{T}_{x_0,s',\theta'}\setminus K\big)$ (see \eqref{ppp3}), we can apply Proposition \ref{maxpri} to get that
\begin{center}
$w_{x_0,s',\theta'}\leqslant  0 \quad $  in $\quad \mathcal{T}_{x_0,s',\theta'}\setminus K$
\end{center}
and therefore in the triangle $\mathcal{T}_{x_0,s',\theta'}$. The desired conclusion


$$w_{x_0,s',\theta'}<  0\qquad \text{in}\quad\mathcal{T}_{x_0,s',\theta'}$$
then follows from \eqref{conf-forte}.

\end{proof}
Now, from small perturbations, we move to larger translations and rotations. We have the following
\begin{lem}[The sliding-rotating technique]\label{largeper}
Let $u\in C^2(\overline{\R^2_+})$ be a {\rm nonnegative} solution to \eqref{E:P}
 and assume that $f$ fulfills ($h_2$) and  ($h_3$).
Let $(x_0,s,\theta )$ be as above and assume that $(\mathcal H\mathcal T_{x_0,s,\theta})$ holds.
 Let $(\hat{s},\hat{\theta})$ be fixed 
and assume that there exists a continuous function $g(t)=(s(t),\theta (t))\, :\,[0\,,\,1] \rightarrow (0, +\infty) \times (0\,,\,\frac{\pi}{2})$, such that
$g(0)=(s,\theta)$ and $g(1)=(\hat{s},\hat{\theta})$.
Assume that
\begin{equation}\label{tttt}
 w_{x_0,s(t),\theta(t)}\leqslant 0\qquad \text{on}\quad\partial(\mathcal{T}_{x_0,s(t),\theta(t)})\qquad
 \text{for every} \quad t\in [0,1)
 \end{equation}

and that
$w_{x_0,s(t),\theta(t)}$ is not identically zero on $\partial(\mathcal{T}_{x_0,s(t),\theta(t)})$  for every $t\in [0,1)$.\\

Then
\begin{center}
$(\mathcal H\mathcal T_{x_0,\hat s,\hat \theta})$ holds.
\end{center}
If $u$ is positive, the same result holds assuming only ($h_4$) instead of ($h_3$).
\end{lem}

\begin{proof}
By the assumptions and exploiting Lemma \ref{smallperturbations} we  obtain the existence of  $\tilde{t}>0$ small such that, for  $0\leqslant t\leqslant \tilde{t}$, $(\mathcal H\mathcal T_{x_0,s(t),\theta(t)})$ holds.


We now set

\begin{center}
$\overline{T}\equiv \,\, \{\tilde{t}\in [0,1]\,\, s.t.\, (\mathcal H\mathcal T_{x_0,s(t),\theta(t)})\,\,\text{  holds for any}\quad 0 \leqslant t\leqslant\tilde{t}\}$
\end{center}
 and $$\bar{t}=\sup\,\overline{T}\,.$$
We claim that actually $\bar{t}=1$. To prove this, assume  $\bar{t}<1$ and note that in this case we have
$$
w_{x_0,s(\bar{t}),\theta(\bar{t})}\le  0 \quad \text{in}\quad  \mathcal{T}_{x_0,s(\bar{t}),\theta(\bar{t})}\,
$$
$$
w_{x_0,s(\bar{t}),\theta(\bar{t})} \le 0 \quad on \quad \partial(\mathcal{T}_{x_0,s(\bar{t}),\theta(\bar{t})})
$$

by continuity, and that $w_{x_0,s(\bar{t}),\theta(\bar{t})}$ is not identically zero on $\partial(\mathcal{T}_{x_0,s(\bar{t}),\theta(\bar{t})})$  by assumption.

Hence, by \eqref{conf-forte}, we see that
$$
w_{x_0,s(\bar{t}),\theta(\bar{t})}< 0 \quad \text{in}\quad  \mathcal{T}_{x_0,s(\bar{t}),\theta(\bar{t})}\,.
$$
Therefore $(\mathcal H\mathcal T_{x_0,s(\bar{t}),\theta(\bar{t})})$ holds and using once again Lemma \ref{smallperturbations}, we can find a sufficiently small $\varepsilon >0$ so that  $(\mathcal H\mathcal T_{x_0, s(t),\theta(t)})$ holds for any
$0\leqslant t \leqslant \bar{t}+\varepsilon$, which contradicts the definition of $\bar{t}$.
\end{proof}

\section{Further preliminary results} \label{sjjkdkgbnnbnbnbnbn}

In this section we assume that $u$ is a nonnegative and non trivial solution of \eqref{E:P}, i.e., $ u \not \equiv 0$.  We also suppose that $f$ fulfills ($h_1$) and, when $ f(0)=0$, also that $f$ fulfills ($h_3$).


\medskip

Given  $x_0 \in \R$, let us set
\begin{equation}\label{QH}
B^+_r(x_0)=B_r(x_0)\cap\{y> 0\}\,
\end{equation}
where $B_r(x_0)$ denotes  the two-dimensional open ball centered at $(x_0,0)$ and of radius $ r>0$.

We claim that, for some $\bar r>0$ and for some $\bar \theta=\bar \theta(\bar r) \in (0, \frac{\pi}{2}) $,

\begin{equation}\label{ccvv}
\frac{\partial\,\,u}{\partial V_\theta} > 0\qquad \text{in }\quad B^+_{\bar r}(x_0)\qquad\text{for}\,\,\,\, - \bar\theta \leq\theta\leq\bar\theta\,.
\end{equation}

\noindent \emph{Case 1: $f(0)<0$}.\\
\noindent Since $\partial_{xx}u(x_0,0) = 0$, we have that
\begin{equation}\nonumber
-\partial_{yy}u(x_0,0)\,=\,-\Delta u (x_0,0)=f(u(x_0,0))=f(0)<0\,.
\end{equation}
 Recalling that $u\in C^2(\overline{\mathbb{R}^2_+})$, we conclude that we can take $\bar r>0$ small such that
\begin{equation}\nonumber
 \partial_{yy}u>0\qquad \text{in }\quad B^+_{\bar r}(x_0).
\end{equation}
Exploiting again the fact that $u\in C^2(\overline{\mathbb{R}^2_+})$, we  can consequently deduce  that
\begin{equation}\label{monteta}
\frac{\partial}{\partial V_\theta}  \left(\frac{\partial\,\,u}{\partial V_\theta} \right)>0\qquad \text{in }\quad B^+_{\bar r}(x_0)\qquad\text{for}\,\,\,\, - \bar\theta \leq\theta\leq\bar\theta\,.
\end{equation}
Also, since we assumed that $u$ is nonnegative in $\mathbb{R}^2_+$, it follows that
\begin{equation}\label{sbvdvjkdbvdvbjkd}
\frac{\partial\,\,u}{\partial V_\theta}(x,0) \geqslant 0\qquad\qquad\text{for any}\,\,\,\, -\bar \theta \leq\theta\leq\bar\theta
 \quad \text{and for any}\,\,\,x\in\R\,.
\end{equation}
Combining \eqref{monteta} and \eqref{sbvdvjkdbvdvbjkd}, we deduce \eqref{ccvv}.\\

\noindent \emph{Case 2: $f(0)\geq0$ and $ u \not \equiv 0$.} \\
In this case we first observe that

\begin{equation}\label{positiva}
\begin{cases}
u >0 \qquad \text{in }\quad \mathbb{R}^2_+,\\
\frac{\partial u}{\partial y} (x,0) > 0  \qquad \forall \, x \in \mathbb{R}.
\end{cases}
\end{equation}
Indeed, if $f(0)>0$ and $u(\overline{x}) = 0 $ with $\overline{x} \in \mathbb{R}^2_+$, then $\Delta u \le 0$ in an open connected neighbourhood of $\overline{x} $ by the continuity of $u$ and $f$. The classical maximum principle and a standard  connectedness argument imply that $ u \equiv 0 $ on $\mathbb{R}^2_+$. The latter contradicts the non triviality of $u$. Therefore $ u>0$ everywhere and the classical Hopf lemma provides the second claim in \eqref{positiva}.

To treat the case $f(0)=0$ we follow the arguments leading to \eqref{eq-conf-forte} and \eqref{conf-forte}.
More precisely, when $f(0)=0$ and $u(\overline{x}) = 0 $, the continuity of $u$ and \eqref{prop-conf-forte} with $\bar t =0$ tell us that $u$ satisfies the inequality $ \Delta u = -f(u) = f(0) -f(u) \le g(u) $ in an open connected neighbourhood of $\overline{x}$.  Then, since $(h_3)$ is in force, we can use the strong maximum principle and the Hopf boundary lemma (see for instance Chapter 5 in \cite{PSB}) to get \eqref{positiva}, as before.

\smallskip
\begin{rem}
Note that, in the previous argument, we used assumption $(h_3)$ only for $\bar{t} =0$ (and only in the case $f(0)=0$).
\end{rem}
\medskip

The desired conclusion \eqref{ccvv} then follows immediately from \eqref{positiva} and the $C^2$-regularity of $u$ up to the boundary.



\medskip

From the  analysis above, we find the existence of (possible very small)

\begin{equation}\label{mons}
\bar s\,=\,\bar s(\bar \theta)>0\,,
\end{equation}
such that, for any $0< s\leqslant \bar{s}$ :

\begin{itemize}
 \item[$i)$] both the triangle $\mathcal{T}_{x_0,s,\bar\theta}$ and its reflection w.r.t. $L_{x_0,s,\bar\theta}$ are contained in $B^+_r(x_0)$ (as well as their reflections w.r.t. the axis $ \{\, x = x_0 \, \}$),
\item[$ii)$] both the segment $\{\, (x_0,y) \, : \,  0 \le y \le s  \, \}$ and its reflection w.r.t. $L_{x_0,s,\theta}$ are contained in $B^+_r(x_0)$ for every $\theta \in (0, \bar\theta]$,
\item[$iii)$] $u < u_{x_0,s,\bar\theta}$ in $\mathcal{T}_{x_0,s,\bar\theta}$,
\item[$iv)$]  $u \leqslant u_{x_0,s,\theta}$ on $\partial(\mathcal{T}_{x_0,s,\theta})$ for every $\theta \in (0, \bar\theta]$,
\item[$v)$]  $u < u_{x_0,s,\theta}$ on the set $ \{\, (x_0,y) \, : \,  0 < y < s  \, \}$, for every $\theta \in (0, \bar\theta]$.
\end{itemize}





Note that, from $iii)-iv)$, we have that
\begin{equation}\label{Hinizio}
\forall \, s \in (0, \bar s), \qquad \mathcal (\mathcal H\mathcal T_{x_0,s,\bar\theta}) \qquad holds.
\end{equation}

\medskip

To continue the description of our results, we denote by $p\,:=\,(x,y)$ a general point in the plane and, for a nonnegative solution $u$ of \eqref{E:P}, we say that $u$ satisfies the property $(\mathcal P_\mu)$ if {\it there exists a real number $ \mu>0$ and a point  $p\in\{y=\mu\}$ such that $u(p)\neq 0$. }

Equivalently :
\begin{equation}\nonumber
(\mathcal P_\mu)\qquad\text{holds if}\qquad\{y=\mu\}\cap\{u\neq0\}\neq\emptyset\,.
\end{equation}

For a non trivial $u$ and under the assumptions stated at the beginning of this section we have that the set

\begin{equation}\label{LAMBDA}
\Lambda^*=\Lambda^*(u)\,:=\,\{\lambda>0\,\,:\,\,(\mathcal P_\mu)\quad\text{holds for every} \,\, 0<\mu\leq\lambda\}
\end{equation}

is not empty. The latter claim follows from \eqref{positiva} when $ f(0) \ge 0$ and from Theorem 6.1 of \cite{FS0} when $f(0) <0$ (note that Theorem 6.1 of \cite{FS0} holds true for functions $f$ which are {\it only} continuous on $ [0, +\infty)$ and so, it applies in our situation since $(h_1)$ is in force).

Therefore we set
\begin{equation}\label{lambda}
\lambda^*=\lambda^*(u)\,:=\, \sup \Lambda^*\, \in (0, +\infty]
\end{equation}

and also note that, by a continuity argument, if $\lambda^*$ is finite, we get that  $\{y=\lambda^*\}\subseteq \{u=0\}$.

\bigskip

\noindent Next we prove a result that allows to start the moving plane procedure.

\begin{lem}[Monotonicity near the boundary]\label{starting}
Let $u\in C^2(\overline{\R^2_+})$ be a {\rm nonnegative}  and {\rm non trivial} solution to \eqref{E:P}
and assume that $f$ is locally Lipschitz continuous on $[0, +\infty).$

Then there exists $\hat \lambda>0$
such that, for any $0<\lambda\leq \hat\lambda$, we have
\begin{equation}\label{sdhsshkshkdvjbvj}
u<u_\lambda\qquad\text{in}\quad \Sigma_\lambda\,.
\end{equation}
Furthermore
\begin{equation}\label{monotHAT}
\partial_y u >0\qquad\text{in}\quad \Sigma_{\hat\lambda}\,.
\end{equation}

\medskip

If $u$ is positive, the  conclusions above hold when either

\item[i)] $f(0) =0$ and $f$ fulfills ($h_1$),  ($h_2$) and ($h_3$),

or

\item[ii)] $f(0)\neq0$ and $f$ fulfills ($h_1$),  ($h_2$) and ($h_4$).

\end{lem}

\begin{proof}
Let $\bar\theta$ given by \eqref{monteta} and $\, \bar s=\bar s(\bar\theta)$
 as in \eqref{mons}. We showed  that, for any $0<s<\bar s$,
$(\mathcal H\mathcal T_{x_0,s,\bar\theta})$ holds.

We use now Lemma \ref{largeper} as follows: for any fixed $s \in (0, \bar s)$ and $ \theta' \in (0, \bar\theta) $
we consider the rotation
\[
g(t)\,=\,(s(t),\theta (t))\,:=\,(s\,,\,t\theta'+(1-t)\bar\theta)\qquad\quad t\in[0\,,\,1]\,.
\]
Recalling that $(\mathcal H\mathcal T_{x_0,s,\bar\theta})$ holds by \eqref{Hinizio}, we deduce  that also
$(\mathcal H\mathcal T_{x_0,s,\theta'})$ holds.
Therefore, by the fact that $0<\theta'<\bar\theta$ is arbitrary and by a continuity argument, we pass to the limit for $\theta'\rightarrow 0$ and get
\begin{center}
 $u(x,y) \le  u_s(x,y)$ in $\Sigma_{s}\cap \{x\leqslant x_0\}$ for  $0<s<\bar s$.
 \end{center}


The invariance of the considered problem w.r.t. the axis $\{\, x = x_0 \,\} $ enables us to use the same argument to treat the case of negative $\theta$, yielding

\begin{center}
 $u(x,y) \le u_s(x,y)$ in $\Sigma_{s}\cap \{x\geqslant x_0\}$ for  $0<s<\bar s$,
 \end{center}
possibly reducing $\bar s$.

Thus $u(x,y) \le u_s(x,y)$ in $\Sigma_{s}$ for every $ s \in (0,\bar s)$.  The desired conclusion \eqref{sdhsshkshkdvjbvj} then follows by taking $\hat\lambda$ such that $  0 < \hat\lambda < \min \{\bar s, \frac{\lambda^*}{2}\}$. Here we have used in a crucial way that the property $(\mathcal P)_{\lambda}$ holds for every $ \lambda \in (0, \hat\lambda]$, so that the case $u\equiv u_{\lambda}$ in $\Sigma_{\lambda}$ is not possible.

Moreover, when $f$ is locally Lipschitz continuous on $(0, +\infty]$, the function $ u_{\lambda} -u>0$ solves a linear equation of the form $ \Delta  (u_{\lambda} -u) = c(x) (u_{\lambda} -u)$, with $c$ locally bounded on $\Sigma_{\lambda}$. Therefore, by the Hopf's Lemma, for every $ \lambda \in (0, \hat\lambda]$ and every $x \in \R$, we get
\begin{equation}\label{Hopf-movingplane}
- 2\partial_y u(x, \lambda)=\frac{\partial (u_{\lambda}-u)}{\partial y}(x, \lambda)<0\,.
\end{equation}
The latter proves \eqref{monotHAT} when $f$ is locally Lipschitz continuous on $(0, +\infty]$.

If $u$ is everywhere positive, \eqref{Hopf-movingplane} is still true since $(h_3)$ (resp. $(h_4)$) is in force. Indeed, the arguments already used to prove \eqref{conf-forte} and \eqref{positiva} and the crucial fact that $ u(x,\lambda)  > 0 $ for every $x \in \mathbb{R}$ lead to

\begin{equation}\label{Hopf-mp2}
\begin{cases}
\Delta ( u_{\lambda}-u) \leq g(u_{\lambda}-u) & \qquad\text{in}\quad B_x, \\
u_{\lambda}-u > 0 & \qquad\text{in}\quad B_x, \\
u_{\lambda}(x,\lambda) - u(x,\lambda)  = 0 & \qquad \forall \, x \in \mathbb{R},
\end{cases}
\end{equation}
where $B_x \subset \mathbb{R}^2$ is an open ball centered at $(x,\lambda)$. Therefore, since $(h_3)$ (resp. $(h_4)$) is in force, the boundary lemma gives \eqref{Hopf-movingplane}. This concludes the proof.

\end{proof}

\medskip

\begin{rem}
Note that, when $u$ is positive, we used only  $(h_4)$.
\end{rem}

\medskip

Let $\lambda^*$ be defined as in \eqref{lambda}.
 In the case $\lambda^*=\infty$ we set
$$\Lambda=\{\lambda>0\, :\,u<u_{\lambda '}\quad\text{in}\,\,\Sigma_{\lambda'}\,\,\,\,\forall \lambda'<\lambda\}\,.$$
If $\lambda^*$ is finite we use the same notation but considering values of $\lambda$ such that $0<\lambda<\lambda^*/2$, namely
$$\Lambda=\{\lambda<\frac{\lambda^*}{2}\, :\,u<u_{\lambda '}\quad\text{in}\,\,\Sigma_{\lambda'}\,\,\,\,\forall \lambda'<\lambda\}\,.$$
By Lemma \ref{starting} we know that $\Lambda$ is not empty and we can define
\begin{equation}\label{hhhhhhhhhh}
\bar{\lambda}=\sup\,\,\Lambda\,.
\end{equation}

Now we assume that $\bar\lambda<+\infty$, when $ \lambda^* = \infty$ (resp. $\bar\lambda<\frac{\lambda^*}{2}$, when $\lambda^*$ is finite) and observe that, arguing as above and under the same assumptions of Lemma \ref{starting} (cfr. the proof of \eqref{Hopf-movingplane}), we deduce that

\begin{equation}
u < u_{\bar{\lambda}} \qquad\text{on}\quad\Sigma_{\bar{\lambda}},
\end{equation}

\begin{equation}\label{monot-limite}
\partial_y u(x,\lambda) \,>\,0\qquad \forall \, (x, \lambda) \in \mathbb{R} \times (0, \bar{\lambda}].
\end{equation}




and then we can prove the following

\begin{lem}\label{caffbe} Let $u$ and $f$ as in Lemma \ref{starting}. Let $\lambda^*$ and $\bar\lambda$ be as above. Assume that there is a point  $x_0\in\mathbb{R}$ satisfying $u(x_0,2\bar\lambda)>0$. Then there exists $\bar \delta >0$ such that:
for any $-\bar\delta\leqslant \theta\leqslant\bar\delta $ and for any
$0< \lambda\leqslant\bar{\lambda}+\bar\delta$, we have
 $$u(x_0,y)<u_{x_0,\lambda,\theta}(x_0,y)\,,$$
for $0< y<\lambda$.
\end{lem}

\begin{proof}
First we note that, by \eqref{monot-limite}, we have $\partial_y u (x_0,\bar{\lambda})>0$.


We argue now by contradiction. If the lemma were false, we found a sequence of small $\delta _n\rightarrow 0$
and
$-\delta _n \leqslant\theta_n\leqslant \delta _n$,
 $0<\lambda_n\leqslant \bar{\lambda}+\delta _n$,
 $0< y_n<  \lambda_n$ with $$u(x_0,y_n)\geqslant u_{x_0,\lambda_n ,\theta_n}(x_0,y_n).$$

Possibly considering subsequences,  we  may and do assume that $\lambda_n\rightarrow \tilde{\lambda}\leqslant \bar{\lambda}$. Also
$y_n\rightarrow \tilde{y}$ for some $\tilde{y}\leqslant \tilde{\lambda}$. Considering the construction of $B^+_{\bar r}(x_0)$ as above and in particular taking into account \eqref{monteta} and \eqref{sbvdvjkdbvdvbjkd}, we deduce that $\tilde\lambda>0$ and,
by continuity, it follows that $u(x_0,\tilde{y})\geqslant u_{\tilde{\lambda}}(x_0,\tilde{y})$.
Consequently
$y_n\rightarrow \tilde{\lambda}=\tilde{y}$, since  we know that $u<u_{\lambda '}$  in $\,\Sigma_{\lambda '}$ for any $\lambda'\leqslant \bar{\lambda}$ and
we assumed that $u(x_0,2\bar\lambda)>0$ so that in particular
$u(x_0,0)=0<u(x_0,2\bar\lambda)$.
By the mean value theorem since $u(x_0,y_n)\geqslant u_{x_0,\lambda_n ,\theta_n}(x_0,y_n)$, it follows
$$\frac{\partial u}{\partial V_{\theta_n}}(x_n,y_n)\leqslant 0$$
at some point $\xi_n\equiv (x_n,y_n)$ lying on  the line from $(x_0,y_n)$ to $T_{x_0,\lambda_n,\theta_n}(x_0,y_n)$, recalling that
the vector $V_{\theta_n}$ is orthogonal to the line $L_{x_0,\lambda_n,\theta_n}$.  Since $V_{\theta_n}\rightarrow e_2$ as $\theta_n\rightarrow 0$.\\
Taking the limit it follows
$$\partial_y u (x_0,\tilde{\lambda})\leqslant 0$$
which is impossible by \eqref{monot-limite}.
\end{proof}

\section{Proof of Theorem \ref{teo-doppio}} \label{fine-teo-doppio}

\begin{proof}[\underline{Proof of Theorem \ref{teo-doppio}}]

Since we are assuming that $\bar\lambda<+\infty$, when $ \lambda^* = \infty$ (resp. $\bar\lambda<\frac{\lambda^*}{2}$, when $\lambda^*$ is finite), by definition of $\lambda^*$ we can find $x_0 \in \R$ such that $u(x_0,2\bar\lambda)>0$. Let $B^+_{\bar r}(x_0)$ be constructed as above and pick $\bar\theta$ given by \eqref{monteta}.

Let also $\bar \delta $ as in Lemma \ref{caffbe}. Then fix $\theta_0>0$ with $\theta_0\leqslant \bar \delta$ and $\theta_0\leqslant\bar\theta$. Let us set
\[
s_0\,:=\,s_0(\theta_0)\,,
\]
such that  the triangle $\mathcal T_{x_0,s_0,\theta_0}$ and its reflection w.r.t. $L_{x_0,s_0,\theta_0}$ is contained in $B^+_{\bar r}(x_0)$ and consequently $(\mathcal H\mathcal T_{x_0,s_0,\theta_0})$ holds. It is convenient to assume that $s_0\leqslant \hat\lambda$  with $\hat\lambda$ as in Lemma \ref{starting}.
For any
\[
s_0<s\leqslant\bar\lambda+\bar\delta,\qquad 0<\theta<\theta_0\,,
\]
we carry out the \emph{sliding-rotating technique} exploiting
Lemma \ref{largeper} with
\[
g(t)\,=\,(s(t),\theta (t))\,:=\,(ts+(1-t)s_0\,,\,t\theta+(1-t)\theta_0)\qquad\quad t\in[0\,,\,1]\,.
\]
By Lemma \ref{caffbe} we deduce that the boundary conditions required to apply Lemma \ref{largeper} are fulfilled and therefore, by Lemma \ref{largeper}, we get that $(\mathcal H\mathcal T_{x_0,s,\theta})$ holds.  We can now argue as in the proof of Lemma \ref{starting} and deduce that
 $u(x,y)<u_\lambda(x,y)$ in $\Sigma_{\lambda}$  for any $0<\lambda\leqslant \bar\lambda+\bar\delta$. This provides a contradiction unless $\bar\lambda=+\infty$
 (resp.
$\bar\lambda=\,\frac{\lambda^*}{2}$, if $\lambda^*$ is finite). Arguing e.g. as in the proof of Lemma \ref{starting}, we deduce
 \[
\partial_y u >0\qquad\text{in}\quad\mathbb{R}^2_+\qquad\,\,\,\text{if}\quad\lambda^*=+\infty\,,
 \]
while
 \begin{equation}\nonumber
\partial_y u > 0\qquad \text{in}\quad\Sigma _{\frac{\lambda^*}{2}}\qquad\text{if}\quad\lambda^*<+\infty\,.
\end{equation}
As a consequence of the monotonicity result, we deduce that $u$ is positive in $\mathbb{R}^2_+$ if $\lambda^*=+\infty$.\\

\noindent Let us now deal with the case when  $\lambda^*$ is finite, that may occur only in the case $f(0)<0$.
We deduce by continuity that
\[
u\leq u_{\lambda^*/2} \qquad \text{in}\quad\Sigma_{\lambda^*/2}\,.
\]
By the strong comparison principle, we deduce that: either  $u< u_{\lambda^*/2}$ or $u\equiv u_{\lambda^*/2}$, in $\Sigma_{\lambda^*/2}$.
Note that, by the definition of $\lambda^*$, we have that $\{y=\lambda^*\}\subseteq \{u=0\}$, that also implies $\{y=\lambda^*\}\subseteq \{\nabla u=0\}$ since $u$ is nonnegative. If $u< u_{\lambda^*/2}$  in $\Sigma_{\lambda^*/2}$,
we get by the Hopf's boundary Lemma (see \cite{GT}) that $\partial_y (u_{\lambda^*/2}-u)>0$  on $\{y=0\}$. Since $\partial_y (u_{\lambda^*/2})=0$ on $\{y=0\}$ (by the fact that $\{y=\lambda^*\}\subseteq \{\nabla u=0\}$) this provides a contradiction with the fact that $u$ is nonnegative. Therefore it occurs $u\equiv u_{\lambda^*/2}$, in $\Sigma_{\lambda^*/2}$. \\

Note now that, since
$
\{y=\lambda^*\}\subseteq \{u=0\}\cap\{\nabla u=0\}\,,
$
by symmetry we deduce
\[
\{y=0\}\subseteq \{u=0\}\cap\{\nabla u=0\}\,.
\]
Therefore we deduce that $u$ is one-dimensional by the \emph{unique continuation principle} (see for instance Theorem 1 of \cite{FVb} and the references therein). Here we use in a crucial way the fact that $f$ is locally Lipschitz continuous on $[0, +\infty)$. Indeed, for every $t \in \R$, the function $u^t(x,y) : = u(x+t,y) $ is a nonnegative solution of \eqref{E:P} with $ u^t = \nabla u^t = 0$ on $ \partial \R^2_+$ and the unique continuation principle implies that $ u \equiv u^t $ on $\R^2_+$. This immediately gives that $u$ depends only on the variable $y$, i.e.,
\[
u(x,y)\,=\,u_0(y) \qquad \forall \, (x,y) \in \R^2_+
\]

where $u_0 \in C^2([0, +\infty)) $ is the unique solution of $u_0''+f(u_0)=0$ with $u_0'(0)=u_0(0)=0$.

The remaining part of the statement, namely the properties of $u_0$, follows by a simple ODE analysis.
\end{proof}

\section{Proof of Theorem \ref{casteorem} and Theorem \ref{T:3spe}}

\begin{proof}[\underline{Proof of Theorem \ref{casteorem}}]
Since $u$ is positive we immediately have that  $\lambda^*=\infty$.
We now observe that, when $u$ is positive, the first part of the proof of Theorem \ref{teo-doppio} holds under the assumptions of Theorem \ref{casteorem} (see Lemma \ref{largeper}, Lemma \ref{starting} and Lemma \ref{caffbe}).
Therefore, arguing as in Theorem \ref{teo-doppio}, we get
\[
\partial_y u >0\qquad\text{in}\quad\mathbb{R}^2_+\, .
 \]
\end{proof}

\begin{proof}[\underline{Proof of Theorem \ref{T:3spe}}]
The proof follows arguing exactly as in the proof of Theorem \ref{casteorem}, just observing that the translating rotating technique can be performed
until we reach the maximal position at the middle of the strip. This provides the fact that
 $u$ is strictly monotone increasing in $\Sigma_{b}$.
 To prove that
$
\frac{\partial u}{\partial y} >0$ in  $\Sigma_{b}$  just argue again as in the  proof of \eqref{Hopf-movingplane} (see also \eqref{conf-forte}).
If $u\in C^2(\overline{\Sigma_{2b}})$
and $u=0$ on $\partial \Sigma_{2b}$, then the technique can be applied in the opposite direction thus proving that
 $u$ is  symmetric about $\{y=b\}$.
\end{proof}

\section{Proof of Theorem \ref{teo-sim}, Corollary \ref{CorClas} and Theorem \ref{teo-clas2}}

\begin{proof}[\underline{Proof of Theorem \ref{teo-sim} }]
Since $f$ is locally Lipschitz continuous on $[0,+\infty)$, Theorem \ref{teo-doppio} implies that, either $u$ is one-dimensional and periodic (possibly identically equals to zero) and in this case we are done, or $u>0$ and $ \frac{\partial u}{\partial y} >0$ everywhere in $\R^2_+$. To conclude  the proof it remains to consider the second case. First we observe that $u$ is necessarily bounded on $\R^2_+$, indeed by Theorem 2.1 of \cite{PQS} there is a positive constant $C$, depending only on $p,f$ and the euclidean dimension, such that
\begin{equation}\label{stima}
u(x,y)  
\le C ( 1 + dist^{-\frac{2}{p-1}}((x,y)\, ; \, \partial \R^2_+)) =
C ( 1 + y^{-\frac{2}{p-1}}) \qquad \forall \, (x,y) \in \R^2_+
\end{equation}
and therefore, the boundedness of $u$ follows by combining the monotonicity of $u$, i.e., $\frac{\partial u}{\partial y} >0$  on $\R^2_+$, together with the estimate \eqref{stima}. Then, by standard elliptic estimates, we also get that $\vert \nabla u \vert $ is bounded and so we can apply Theorem 1.6 of our previous work \cite{FS0} to get that $u$ is one-dimensional. This concludes the proof.
\end{proof}

\begin{proof}[\underline{Proof of Corollary \ref{CorClas}}]
If $u$ is a solution to \eqref{pblA}, by the strong maximum principle we have either $u \equiv 0$ or $u>0$.
Then, by proceeding as in the proof of Theorem \ref{teo-sim} we get that  either $u \equiv 0$ or $u$ is a positive, bounded, one-dimensional and monotonically increasing function, say $u = u(y)$. The second case is impossible since $\bar l := \lim_{y\to +\infty} u(y)$ would be a positive zero of $f$, contradicting the assumption \eqref{posit}.  Therefore, if $u$ is a solution then necessarily $ u \equiv 0$ and so $f(0)=0$. This completes the proof.
\end{proof}

\begin{proof}[\underline{Proof of Theorem \ref{teo-clas2} }]
By assumption $f\ge0$ on $[0, +\infty)$ and so, either $u \equiv 0$ or $u>0$, thanks to the strong maximum principle. Since the linear function $ u(x,y) = cy$, $c \ge 0$, is harmonic, to conclude the proof of the first claim it is enough to show that $ u>0 \Longrightarrow   u(x,y) = cy$ for some $c>0$ (which in turn implies that $f \equiv 0$).
To this end we first observe that
\begin{equation}\label{v-pos}
v := \partial_y u > 0 \qquad {\textit {on}} \quad \overline{\R^2_+},
\end{equation}
indeed, $u >0$ on $\R^2_+$ implies $ v>0$ on $\R^2_+$ by Theorem \ref{teo-doppio} and $v>0$ on $\partial \R^2_+$ by Hopf's Lemma, since $u>0$  on $\R^2_+$ and $f(0) \ge 0$ by assumption.

Then we remark that, for every  $r>0$, for every $ p \in \overline{\R^2_+}$ and every
open ball $B_r(p)$,
\begin{equation}\label{u-H3}
u \in H^3(B_r(p) \cap \R^2_+)
\end{equation}
by standard elliptic regularity (see, for instance, Theorem 8.13 of \cite{GT}) and so we get that
\begin{equation}\label{v-H2}
v \in C^1(\overline{\R^2_+}), \qquad v\in H^2(B_r(p) \cap \R^2_+) \qquad \forall r>0, \quad \forall  p \in \overline{\R^2_+}.
\end{equation}
Now, since $f$ is locally Lipschitz continuous,  by differentiating the equation satisfied by $u$ and using \eqref{v-pos} and \eqref{v-H2} we obtain that $v$ satisfies
\begin{equation}
\begin{cases}
\quad v > 0& \text{everywhere in $\,\, {\overline {\R^2_+}}$},\\
-\Delta v = f'(u)v \ge 0 & \text{a.e. in $\,\, \R^2_+$}.
\end{cases}
\end{equation}
Then, for any $ \psi \in C^{0,1}_c(\R^2)$, we multiply the latter equation by $\psi^2 v^{-1}$ and integrate by parts to get
\[
0 \le - \int_{\R^2_+} \Delta v \psi^2 v^{-1} =
\]
\[
- \int_{\R^2_+} \frac{\vert \nabla v \vert^2}{v^2} \psi^2 + \int_{\R^2_+} 2 v^{-1} \psi \nabla v \nabla \psi + \int_{\partial \R^2_+} \frac{\partial v}{\partial y} v^{-1} \psi^2.
\]
Now we observe that $\frac{\partial v}{\partial y} = \frac{\partial^2 u}{\partial y^2} = \Delta u = -f(0) \le 0 $ on $\partial \R^2_+$, by the assumption \eqref{positzero}, and therefore  we deduce from the latter that
\begin{equation}
\int_{\R^2_+} \frac{\vert \nabla v \vert^2}{v^2} \psi^2 \le  \int_{\R^2_+} 2 v^{-1} \psi \nabla v \nabla \psi
\end{equation}
and then
\begin{equation}
\int_{\R^2_+} \frac{\vert \nabla v \vert^2}{v^2} \psi^2 \le  2 \Big [ \int_{\R^2_+} \frac{\vert \nabla v \vert^2}{v^2} \psi^2 \Big ]^{\frac{1}{2}} \Big [ \int_{\R^2_+} \vert \nabla  \psi \vert ^2 \Big ]^{\frac{1}{2}}
\end{equation}
which gives
\begin{equation}\label{v-Liouville}
\int_{\R^2_+} \frac{\vert \nabla v \vert^2}{v^2} \psi^2 \le  4 \int_{\R^2_+} \vert \nabla  \psi \vert ^2.
\end{equation}

Now, for every $R > 1$, consider the functions $\psi_R \in C^{0,1}_c(\R^2)$ given by

\[
\psi_R (x):= {\textbf 1}_{B_{\sqrt R}}(x)+
{{2\ln (R/\vert x \vert)} \over {\ln R}} {\textbf 1}_{B_R\setminus B_{\sqrt R}}(x), \qquad \forall x \in \R^2
\]
and used them into \eqref{v-Liouville} to get
\begin{equation}\label{v-Liouville2}
\int_{\R^2_+} \frac{\vert \nabla v \vert^2}{v^2} \psi_R^2 \le  4 \int_{\R^2_+} \vert \nabla  \psi_R \vert ^2 \le \frac{C}{\log R},
\end{equation}
where $C$ is a positive constant independent of $R$. By letting $ R \longrightarrow +\infty$ in \eqref{v-Liouville2} we find $\int_{\R^2_+} \frac{\vert \nabla v \vert^2}{v^2} =0$ and so $ v $ is a positive constant, say $  v \equiv c>0$. This means that $\partial_y u \equiv c >0$ and so $u(x,y) = cy $ for every $(x,y) \in \R^2_+$. Therefore $ 0 = - \Delta (cy) = f(cy) $ for every $y>0$ and thus $ f \equiv 0$, since $c>0$. This concludes the proof of the first claim.

In view of the  discussion above, if $ f \not \equiv 0$, the only solution of \eqref{pblB} is $ u \equiv 0$, which immediately implies item i) and item ii).


\end{proof}

\section{A counterexample} \label{preliminaries}

In this section we provide a counterexample showing that the conclusion of Theorem \ref{casteorem} (and of Theorem \ref{T:3spe}) fails if $f$ satisfies $(h_1),(h_2)$ but not $(h_3)$, i.e., the monotonicity property $\frac{\partial u}{\partial y} >0$ in $\R^2_+$, does not hold true if $f$ satisfies $(h_1),(h_2)$ but not $(h_3)$.
To this end we shall follow section 6 of \cite{Fa1}.  With the notation of example 6.3 of \cite{Fa1}, the function
\begin{equation}\label{esempio}
u(x,y) := \begin{cases}
u_1(x,y),& \quad y \le 2,\\
v(x,y -5), & \quad y > 2,
\end{cases}
\end{equation}
given by formula $(6.9)$ on p. 832 of \cite{Fa1}, with $s=0$ and $x_0 =(0,5) \in \R^2$, is a smooth entire solution of the equation $ -\Delta u = h(u)$ in $\R^2$, where $h$ is given by formula $(6.8)$ on p. 832 of \cite{Fa1}. Observe that $u$ is identically zero on the closed affine half-plane $\{ (x,y)  \in \R^2 : y \le -1 \} $ and positive on the open affine half-plane $\{ (x,y)  \in \R^2 : y > -1\},$ therefore the function $v(x,y) := u(x, y-1)$ is a solution of
\begin{equation}
\begin{cases}
-\Delta v = h(v)& \text{in $\,\, \R^2_+$},\\
v>0& \text{in $\,\, \R^2_+$},\\
v=0 & \text{on $\,\, \pa \R^2_+$},
\end{cases}
\end{equation}
which is neither monotone nor one-dimensional. On the other hand $h$ (extended to be equal to the constant $h(2) =0$ for $ t \ge 2$) is a function satisfying $(h_1),(h_2)$ but not $(h_3)$. Indeed, $h$ is Holder continuous on $[0, +\infty)$ and so it satisfies $(h_1)$. Moreover, $h$ fulfills $(h_2)$ since it is non increasing in a neighbourhood of the points $0, 1$ and $2$ and smooth on $[0, +\infty) \setminus \{0,1,2\}$. Finally, let us prove that $h$ does not satisfy assumption $(h_3)$ at ${\bar t}=0$. To this end we first observe that $h(t) =  -192[t(1-t^{\frac{1}{4}})]^{\frac{1}{2}} [1 - {\frac{5}{4}} t^{\frac{1}{4}}]$ for $ t \in [0,1]$ (cf. example 6.1. of \cite{Fa1}) and we suppose, for contradiction, that $h$ satisfies $(h_3)$ at ${\bar t}=0$. Hence, there exists a function $g$ such that
\begin{equation}\label{prop-conf-forte-controes}
h(s)-h(t)\leq g(t-s)\,\qquad \forall \,s\leq t\in [0 \,,\, \frac{\delta}{2}]
\end{equation}
for some $\delta \in (0,1)$ and fulfilling the integral condition \eqref{cond-int-PM}. By choosing $s=0$ in \eqref{prop-conf-forte-controes}, we have
\begin{equation}\label{prop-conf-forte-controes2}
-h(t)\leq g(t)\,\qquad \forall \,  t\in [0 \,,\, \frac{\delta}{2}]
\end{equation}
and, in view of the explicite form of $h$ near zero, we can find $ \eta \in (0, \frac{\delta}{2}),$ small enough, such that
\begin{equation}\label{prop-conf-forte-controes3}
g(t) \ge - h(t) = \vert h(t) \vert \ge \gamma t^{\frac{1}{2}} \,\qquad \forall \,  t\in [0 \,,\, \eta]
\end{equation}
for some $ \gamma >0$.
The latter yields $ G(s) \ge {\frac{2}{3}} \gamma s^{\frac{3}{2}}$ in $ [0, \eta]$ and so $\int_0^\eta\frac{ds}{\sqrt{G(s)}} < \infty$, contradicting \eqref{cond-int-PM}. So, assumption $(h_3)$ is not satisfied at $ {\bar t} =0$ (also note that a similar argument shows that $h$ does not satisfy $(h_3)$ neither at $1$ nor at $2$). Clearly, the same example can be used as a counterexample for Theorem \ref{T:3spe}.

\end{document}